\documentclass[11pt,a4paper,twoside]{amsart}

\usepackage{amsmath,amsfonts,amsthm,amsopn,color,amssymb,enumitem,cite,soul}
\usepackage{palatino}
\usepackage{graphicx}
\usepackage[normalem]{ulem}
\usepackage[colorlinks=true,urlcolor=blue,citecolor=red,linkcolor=blue,linktocpage,pdfpagelabels,bookmarksnumbered,bookmarksopen]{hyperref}
\hypersetup{urlcolor=blue, citecolor=red, linkcolor=blue}
\usepackage[left=3.4cm,right=3.4cm,top=2.72cm,bottom=2.72cm]{geometry}

\usepackage[hyperpageref]{backref}

\usepackage[colorinlistoftodos]{todonotes}
\makeatletter
\providecommand\@dotsep{5}
\def\listtodoname{List of Todos}
\def\listoftodos{\@starttoc{tdo}\listtodoname}
\makeatother

\newcommand{\eps}{\varepsilon}

\newcommand{\R}{\mathbb{R}}

\newcommand{\RN}{{\mathbb{R}^N}}
\newcommand{\RD}{{\mathbb{R}^2}}
\newcommand{\RT}{{\mathbb{R}^3}}

\renewcommand{\le}{\leslant}
\renewcommand{\ge}{\geslant}
\renewcommand{\a }{\alpha }

\renewcommand{\b }{\beta }

\renewcommand{\d }{\delta }

\newcommand{\g }{\gamma }

\renewcommand{\l }{\lambda}

\newcommand{\n }{\nabla }

\newcommand{\G}{\Gamma}

\renewcommand{\H}{H^1(\RD)}

\newcommand{\Ha}{\mathcal H}
\newcommand{\Har}{\mathcal H_r}

\renewcommand{\o}{\omega}

\newcommand{\D }{{\mathcal D}^{1,2}(\RT)}

\newcommand{\ird }{\int_{\RD}}

\newcommand{\irt }{\int_{\RT}}

\def\bbm[#1]{\mbox{\boldmath $#1$}}
\newcommand{\beq }{\begin{equation}}
\newcommand{\eeq }{\end{equation}}

\renewcommand{\le}{\leqslant}
\renewcommand{\ge}{\geqslant}

\newtheorem{theorem}{Theorem}[section]
\newtheorem{lemma}[theorem]{Lemma}

\newtheorem{proposition}[theorem]{Proposition}
\newtheorem{remark}[theorem]{Remark}
\newtheorem{corollary}[theorem]{Corollary}

\title[Planar S-P system with a positive potential]{The planar Schr\"odinger - Poisson system with a positive potential}

\author[A. Azzollini]{Antonio Azzollini}
\address{Dipartimento di Matematica, Informatica ed Economia, Universit\`a degli
	Studi della Basilicata,
	\newline\indent
	Via dell'Ateneo Lucano 10, I-85100
	Potenza, Italy}
\email{antonio.azzollini@unibas.it}

\thanks{}

\begin{document}

	\begin{abstract}
		In this paper we consider the problem 
				\begin{equation*}
		\left	\{
		\begin{array}{l}
			-\Delta u  \pm \phi u +W'(x,u)=0\hbox{ in } \R^2,	\\
			\Delta \phi = u^2 \hbox{ in } \R^2,
		\end{array}
		\right.
		\end{equation*}
 where $W$ is assumed positive. In dimension three, the problem with the sign + (we call it $(\mathcal P_+)$) was considered and solved in \cite{M}, whereas in the same paper it was showed that no nontrivial solution exists if we consider the sign - (say it $(\mathcal P_-)$).
 
 We provide a general existence result for $(\mathcal P_+)$ and two examples falling in the case $(\mathcal P_-)$ for which there exists at least a nontrivial solution.

	\end{abstract}

\maketitle

	\section{Introduction}
	
We are interested in finding finite energy solutions to the following class of problems
		\beq \label{e0}\tag{${\mathcal{P}}_{\pm}$}
	\left	\{
	\begin{array}{l}
		-\Delta u  \pm \phi u +W'(x,u)=0\hbox{ in } \R^2,	\\
		\Delta \phi = u^2 \hbox{ in } \R^2,
	\end{array}
	\right.
	\eeq
where $W=W(x,s):\RD\times\R\to\R_+$ and $W'$ denotes the derivative of $W$ with respect to $s$.\\ 
We will refer to problem \eqref{e0}  by $({\mathcal{P}}_+)$ or $({\mathcal{P}}_-)$ according to the sign before the second term in the first equation.

The attention on this kind of problems has increased in the recent period, starting from the paper of Stubbe \cite{S} where firstly a system of this type was proposed. Actually, in dimension $3$ or larger, the system is well known and there is a wide literature on it. See for example \cite{AR, ADP, AP, DM, R, K}.

In dimension $2$ the system is definitely less known and studied, although its peculiar features and significant differences arising by a comparison with the three dimensional analogous problem, are of great interest. The variational approach developed by Stubbe was used in \cite{CW} and \cite{DW} to study $(\mathcal P_+)$ for a changing sign potential as $W(x,u)=a(x)u^2-b|u|^p$ where $a\in L^{\infty}(\RD,\R_+\setminus\{0\})$ (possibly $a$ is a positive constant) and $b\in \R_+$. Recently, in \cite{WCR} the problem $(\mathcal P_+)$ has been studied in the so called {\it zero mass case}. A class of nonlinearities $W$ have been considered, superquadratic at the origin, growing no more than a power and satisfying the following technical assumption
	\begin{itemize}
		\item[$\mathcal A)$]the function $\frac{W(s)-sW'(s)}{s^3}$ is nondecreasing in $(-\infty, 0)$ and in $(0,+\infty)$.
	\end{itemize}
Up to our knowledge, the literature on problem $(\mathcal P_-)$ is definitely missing.

In this paper we are interested in investigating \eqref{e0} when $W\ge 0$.
The interest in problems in presence of nonnegative potentials $W$ increased in recent years because of  physical considerations on models described by similar equations in three dimensions (see the introduction in \cite{M}). As an example, in \cite{CLL} it was studied an equation strictly related with the problem $(\mathcal P_+)$ in $\RT$, with $W(u)=\frac 1 2(\o u^2 + |u|^{\frac {10}3}) $.  It was introduced to describe a  Hartree model for crystals.\\
Our object is to emphasize analogies and, especially, differences occurring when we consider positive potentials $W$ in \eqref{e0} . \\

Consider firstly the problem $(\mathcal P_+)$.
We introduce the following assumptions on $W$: 
	\begin{itemize}
		\item [$W1)$] $W=W(s)\in C^1(\R)$,
		\item [$W2)$] $W$ is nonnegative and $W(0)=0$,
		\item [$W3)$] there exist $p\in(2,3)$, $C_1$ and  $C_2$ positive constants such that $|W'(s)|\le C_1|s|+C_2|s|^{p-1}$,
		\item [and]
		\item [$W3)'$] there exist $p\in(2,4)$, $C_1$ and  $C_2$ positive constants such that $W(s)\le C_1s^2+C_2|s|^p$, and for any $s\in\R$ we have $0\le W'(s)s\le 4 W(s)$.
	\end{itemize}
	\begin{remark}
		Observe that the assumptions in \cite{WCR}, and in particular $\mathcal A$, do not exclude the possibility that $W$ might be nonnegative (consider, for instance, $W(s)=\frac 25 |s|^{\frac 52}$). However we see that, assuming the following Ambrosetti Rabinowitz condition (which is a something slightly stronger than superlinearity at infinity)
			\begin{itemize}
					\item[$\mathcal AR$)] there exists $\d>1$ and $R>0$ such that $W(s)\d\le W'(s)s$ for $|s|>R$,  
			\end{itemize}  
		assumption $\mathcal A$ implies that $|W'(s)|=O(s^2)$ for $|s|\to +\infty$, which almost corresponds to $W3$. For this reason our conditons are in some sense more general (and less technical) than those in \cite{WCR}, in the case of positive potentials. 
	\end{remark}
The system has a variational structure. Indeed, set $\phi_u= \frac 1 {2\pi} \log|x|\star u^2$ for any $u\in\H$, then solutions of \eqref{eq:e0} can be found looking for critical points of the functional 
 	\begin{equation}
	 	I(u)=\frac 12 \ird |\n u|^2 \, dx +  \frac 1 4 \ird \phi_u u^2\, dx+\ird W(u)\, dx
 	\end{equation}
 which is well defined and $C^1$ in the space
 	\begin{equation*}
	 	E:=\H\cap L^2(\RD, V(x)\,dx),
 	\end{equation*}
where $L^2(\RD, V(x)\,dx)$ is the weighted Lebesgue space with the weight $V(x)= \log(2+ |x|)$, endowed with the norm $\|u\|=\left(\|\n u\|_2^2 + \ird \log (2+|x|) u^2\, dx\right)^{\frac 12}$.
\\
To understand the relation between critical points of $I$ and solutions of  $\mathcal P_+$ we refer to \cite[Proposition 2.3]{CW}.

By the coercivity of $V$, it is a classical result the compact embedding
	\begin{equation*}
	E\hookrightarrow L^q(\RD), \quad \forall q\ge 2.
	\end{equation*} 

Assuming the following usual notations (see \cite{CW, DW, S, WCR}):
	\begin{equation}\label{eq:V}
		\begin{array}{l}
			V_0(u)=\frac 1{2\pi} \ird\ird \log(|x-y|) u^2(x) u^2(y)\, dx\,dy,\\
			V_1(u)=\frac 1{2\pi} \ird\ird \log(2+|x-y|) u^2(x) u^2(y)\, dx\,dy,\\
			V_2(u)=\frac 1{2\pi} \ird\ird \log\left(1+\frac 2{|x-y|}\right) u^2(x) u^2(y)\, dx\,dy,\\
		\end{array}
	\end{equation}
we have 
	\begin{equation*}
		I(u)=\frac 12 \ird |\n u|^2 \, dx +  \frac 1 4V_0(u)+\ird W(u)\, dx
	\end{equation*}	
and $V_0=V_1-V_2$. Moreover, as proved in \cite{S}, for any $u\in E$
	\begin{equation}\label{eq:estV2}
			V_2(u)\le C \|u\|^4_{8/3} \le \tilde C	\|\n u\|_2\|u\|^3_2
	\end{equation}
and then  for any $\eps>0$ there exists $C_\eps >0$ such that 
	\begin{equation}\label{eq:estV22}
		V_2(u)\le \eps \|\n u\|_2^2 + C_\eps \|u\|_2^6.
	\end{equation}
	
	\begin{theorem}\label{main}
		Assume $W1, W2$ and either $W3$ or $W3'$. Then there exists a nontrivial solution to $(\mathcal P_+)$.
	\end{theorem}

		Comparing conditions $W3$ and $W3'$, we observe that the price to obtain a better growth condition at infinity consists in introducing a reversed Ambrosetti Rabinowitz condition. However, because of difficulties in proving the geometry of mountain pass, we are not able to achieve the power four as an admissible growth degree.\\

As to the problem $(\mathcal{P}_-)$, since we were not able to find any reference, let us say few words to introduce the motivation of our study. Actually, the problem is the two-dimension version of what, in three dimension, is known as the Schr\"odinger-Maxwell system in the electrostatic case. To explain the difficulties arising in treating it, we analyze by an example how the problem should present if we assumed the same hypotheses as in Theorem \ref{main}.  

Consider for instance the problem
\begin{equation}\label{eq:e0}
\left	\{
\begin{array}{l}
-\Delta u  -\phi u +|u|^{p-2}u=0\hbox{ in } \R^2,	\\
\Delta \phi = u^2 \hbox{ in } \R^2,
\end{array}
\right.
\end{equation}
and the related formal functional 
\begin{equation}\label{eq:functionall}
I(u)=\frac 12 \ird |\n u|^2\, dx  -  \frac 1 4V_0(u)+\frac 1 p \ird |u|^p\, dx
\end{equation}
well defined and $C^1$ in the space $E$. It is easily seen that the functional is strongly indefinite, and this indefinitness can not be removed by a compact perturbation. An idea to remove the indefiniteness and recover nice geometrical properties is to modify suitably \eqref{eq:e0}. In view of this, we proceed by following two ways: either we introduce a suitable nonautonomous linear perturbation, or we add a sublinear term. 

As regards  the first way, consider  nonlinearities of the type $W(x,s)=\frac 12 K(x)s^2 + \frac 1p |s|^p$, where a suitable growth condition on $K$  to control the term $V_1(u)$ is in order. Actually, since the growth of the weight in the norm of the space $E$ is logarithmic, any power-like function (with positive exponent) is a good candidate to take on the role of $K$. Motivated by these reasons, we introduce the model problem
\beq \label{eq:e1}\tag{$\mathcal{P}_{-}^\a$}
\left	\{
\begin{array}{l}
	-\Delta u  +(1+|x|^\a) u- \phi u +|u|^{p-2}u=0\hbox{ in } \R^2,	\\
	\Delta \phi = u^2 \hbox{ in } \R^2.
\end{array}
\right.
\eeq

We can prove the following result
\begin{theorem}\label{main2}
	Assume $2<p$ and either $\a \in \left(0,\frac{2p-4}{p-4}\right]$ and $p>4$, or $\a >0$ and $2<p\le 4$. Then \eqref{eq:e1} possesses a nontrivial  solution.
\end{theorem}

The second way that we propose to approach $(\mathcal P_-)$ is by assuming that $W'$ is sublinear near the origin.

For the sake of semplicity, we again consider a model problem as a perturbation of \eqref{eq:e0}
	\beq \label{eq:e2}\tag{$\mathcal{P}_{-}^\b$}
	\left	\{
	\begin{array}{l}
		-\Delta u +|u|^{\b-2}u - \phi u +|u|^{p-2}u=0\hbox{ in } \R^2,	\\
		\Delta \phi = u^2 \hbox{ in } \R^2.
	\end{array}
	\right.
	\eeq

\begin{theorem}\label{main3}
	Assume that $\b\in(1,2)$ and $p>2$. Then there exists a nontrivial radial  solution to \eqref{eq:e2}. 
\end{theorem}

 It is really interesting to observe how remarkable is the impact that the dimension has on the problem $(\mathcal P_-)$  under the assumption $W2$. Indeed we have the following nonexistence result, which comes to the same conclusion of \cite[Proposition 1.2]{M} where the nonautonomous case was treated
 
 \begin{theorem}\label{nonex}
 	Assume that $W$ satisfies $W1, W2$ and $W'(0)=0$. If  the couple $(u,\phi)\in H^1(\RT)\times \D$ solves $(\mathcal P_-)$ in $\RT$ and $W(u)\in L^1(\RT)$ then $(u,\phi)=(0,0)$.
 \end{theorem}

\begin{remark}
	As already pointed out, the problem $(\mathcal P_-)$ is completely new and there is a lot to study about it. For example, we are quite confident that by some work it is possible to adapt known techniques to prove multiplicity results for both $\eqref{eq:e1}$ and $\eqref{eq:e2}$ . On the other hand, we also believe that $(\mathcal P_-)$ could be successfully treated in correspondence of more general nonnegative $W$, besides the possibility of studying it for sign-changing $W$ or, possibly, in the harder case of nonpositive $W$.\\
	All these issues are not objects of this paper (whose main purpose is to emphasize the surprising conflict between Theorem \ref{main2} and \ref{main3} on one hand and Theorem \ref{nonex} on the other). We pospone to future papers a deeper study on open problems concerning $(\mathcal P_-)$.
\end{remark}

The paper is organized as follows. 

Section \ref{sec:P_+} is devoted to study the problem $(\mathcal P_+)$ and prove Theorem \ref{main} by means of the mountain pass theorem.

In Section \ref{sec:P_-} we consider problem $(\mathcal P_-)$. It is split in three subsections: the first is aimed to introduce the variational setting and develop the arguments for proving Theorem \ref{main2},  the second is, analagously, completely devoted to the proof of Theorem  \ref{main3} and, finally, in the third, we provide a proof for the nonexistence result  stated in Theorem \ref{nonex}.
 
 \section {The problem $\mathcal P_+$}\label{sec:P_+}
 
 In this section we are going to prove the Theorem \ref{main}, so that in what follows, we are assuming $W1, W2$ and either $W3$ or $W3'$.\\
 
 We will use the following result which is a slightly modified version of \cite[Theorem 1.1]{J}.
 
\begin{theorem}\label{JJ}
	Let $(X,\|\cdot\|)$ be a Banach space, and $J\subset \R_+$ an interval. Consider the family of $C^1$ functionals on $X$
		\begin{equation*}
			I_\l(u)=A(u)-\l B_(u),\quad \l\in J,
		\end{equation*}
	with $B$ nonnegative and either $A(u)\to + \infty$ or $B(u)\to + \infty$ as $\|u\|\to\infty$ and such that $I_\l(0)=0.$
	
	For any $\l\in J$ we set
		\begin{equation}
			\G_\l:=\{\g\in C([0,1],X)\mid \g(0)=0, I_\l(\g(1))<0\}.
		\end{equation} 
	If for every $\l\in J$ the set $\G_\l$ is nonempty and 
		\begin{equation}
			c_\l:=\inf_{\g\in\G_\l}\max_{t\in[0,1]}I_\l(\g(t))>0,
		\end{equation}
	then for almost every $\l\in J$ there is a sequence $(v_n)_n$ in $X$ such that 
	\begin{itemize}
		\item [(i)] $(v_n)_n$ is bounded;
		\item [(ii)] $I_\l(v_n)\to c_\l$;
		\item [(iii)] $(I_\l)'(v_n)\to 0$ in the dual of $X$.
	\end{itemize}
\end{theorem}

In our case $X=E$ and
	\begin{align*}
		A(u)&:= \frac 12\ird |\n u|^2\, dx + \frac 1 4 V_1(u) + \ird W(u)\, dx\\
		B(u)&:= \frac 14 V_2(u).
	\end{align*}
 
 We show that the functional possesses the mountain pass geometry
 
 	\begin{proposition}
	 		There exist $\rho>0$, $\g > 0$ and $\bar u\in E$ such that, said $B_{\rho}$ the ball of radius $\rho$ in $E$,
 	 			\begin{itemize}
	 				\item $I_{|B_{\rho}}\ge 0$ and $I_{|\partial B_{\rho}}\ge \g$,
	 				\item $\|\bar u\|>\rho$ and $I(\bar u)<0$.
	 			\end{itemize}
 	\end{proposition} 
 
 	\begin{proof}
 		Take $u\in E$. Then, for a suitable choice of a small $\eps$ in \eqref{eq:estV22} 
 			\begin{align*}
	 			I(u) &\ge\frac 12 \|\n u\|_2^2 +  \frac 1 4V_1(u) - \frac 1 4 V_2(u)\\
	 				  &\ge \left(\frac 1 2- \frac \eps 4\right)\| \n u\|_2^2 + \frac {\log 2}{8\pi}\|u\|_2^4-\frac {C_\eps}4 \|u\|_2^6
 			\end{align*}
 		which implies the first geometric property.\\
 		As regards the second, observe that if we take $u\in E$, denoting for every $t>0$ $u_t(\cdot)= t^2 u(t\cdot)$, since
 			$$V_0(u_t)= \frac {t^4}{2\pi} \ird\ird \log(|x-y|) u^2(x) u^2(y)\, dx\,dy - \frac{t^4 \log t}{2\pi} \left(\ird  u^2(x)\, dx\right)^2,$$ 
 		we deduce that there exists $\bar t>0$ such that $V_0(u_{\bar t})<0$.\\
 		Now, by $W3$ or, respectively $W3'$, for a sufficiently large $h>0$ we have 
 			\begin{equation*}
	 			I(h u_{\bar t} )\le \frac 12 h^2\|\n u_{\bar t}\|_2^2 + \frac {h^4}4 V_0(u_{\bar t}) + C_1 h^2 \|u_{\bar t}\|_2^2 + C_2 h^p\|u_{\bar t}\|_p^p <0
 			\end{equation*}
 		with $\|h u_{\bar t}\|>\rho$.
 	\end{proof}
 
 	Defining $J=[1-\d,1]$ for $\d>0$, by an easy continuity argument we get the following
 		\begin{corollary}
	 		There exist $\rho>0$, $\g > 0$, $\bar u\in E$ and $\d>0$ such that the following two properties
				\begin{itemize}
					\item ${I_{\l}}_{|B_{\rho}}\ge 0$ and ${I_{\l}}_{|\partial B_{\rho}}\ge \g$,
					\item $\|\bar u\|>\rho$ and ${I_{\l}}(\bar u)<0$,
				\end{itemize}
			hold uniformly for $\l\in[1-\d,1].$
 		\end{corollary}
 
 By this Corollary, the sets $\G_\l$ are nonempty and the mountain pass levels $c_\l$ are well defined and uniformly bounded from below by a positive constant.
 
 Now, exploiting Theorem \ref{JJ}, consider a sequence $\l_n$ in $J$ such that $\l_n\nearrow 1$ and for which there exists a bounded Palais-Smale sequence for $I_{\l_n}$ at the level $c_{\l_n}$.  Then we have the following result
 
 	\begin{proposition}
	 		There exists a sequence $(u_n)_n$  in $E$ such that
	 			\begin{itemize}
	 				\item $I_{\l_n}(u_n)=c_{\l_n}$,
	 				\item $I'_{\l_n}(u_n)=0$,
	 				\item $P_{\l_n}(u_n)=0$,
	 			\end{itemize} 
 			where $P_{\l_n}:E\to \R$ is the functional
 				\begin{equation}\label{eq:poho}
 				P_{\l_n}(v):=V_1(v)-\l_n V_2(v)+\frac 1{8\pi}\left(\ird v^2\,dx\right)^2+ 2 \ird W(v)\,dx .
 				\end{equation} 
 	\end{proposition}
 	\begin{proof}
			Take any $\bar n\ge 1$. We set $\bar \l=\l_{\bar n}$ and consider $ (v_n)_n$ a bounded sequence in $E$ such that $\|I'_{\bar \l}(v_n)\|\to 0$ and $I_{\bar \l}(v_n)\to c_{\bar \l}$.\\
			By boundedness and compact embedding, there exists $v\in E$ such that, up to a subsequence,
				\begin{align*}
					&v_n\rightharpoonup v\hbox{ in } E,\\
					&v_n\to v\hbox{ in } L^q(\RD),\;\forall q\ge 2.
				\end{align*}
			Now, we proceed as in the final part of \cite[Proof of  Prooposition 3.1]{CW} (actually the proof  is easier since we do not need to translate the sequence) and deduce that $v_n\to v$ in $E$ (observe that we already know that $v_n\to v$ in $L^2(\RD)$). As a consequence we have $I'_{\bar \l}(v)=0$ and $I_{\bar\l}(v)=c_{\bar\l}$.\\
			To see that $P_{\bar\l}(v)=0$ we first recall that, since for every $u\in H^1(\RD)$ the function $\phi_u\in L^{\infty}_{loc}(\RD)$ (see \cite{CW}), we can deduce by \cite[Theorem 8.8]{GT} that $v$ is in $W^{2,2}_{loc}(\RD)$ and use standard regularity arguments to show that $v\in C^2(\RD)$. Then $P_{\bar\l}(v)=0$ follows as in \cite{DW}.
		
	\end{proof}
 
Now we are ready to prove our result
 
 	\begin{proof}[Proof. of Theorem \ref{main}]
		Consider $(u_n)_n$ as in the previous Proposition. We are going to prove that it is bounded in $E$. For making the notation less cumbersome, we will write simply $n$ in the place of $\l_n$ as a subscript.\\
		Indeed we have
			\begin{align}
				&\frac 12 \ird |\n u_n|^2\, dx +  \frac 1 4V_1(u_n)+\ird W(u_n)\, dx-\frac{\l_n}{4}V_2(u_n)=c_n+o_n(1),\label{eq:func}\\
				&\ird |\n u_n|^2 \, dx +  V_1(u_n)+\ird W'(u_n)u_n\, dx-\l_n V_2(u_n)=0,\label{eq:neh}\\
				&  V_1(u_n)+\frac 1{8\pi}\left(\ird u_n^2\,dx\right)^2+2\ird W(u_n)\, dx-\l_nV_2(u_n)=0.\label{eq:poho2}
			\end{align}
		
		Assume $W3$. Dividing \eqref{eq:poho2} by $\frac 18$ and summing \eqref{eq:func} we have
			\begin{multline}\label{eq:mix}
				\frac 12 \ird |\n u_n|^2\, dx + \frac 3 8V_1(u_n)+\frac 1{64\pi}\left(\ird u_n^2\,dx\right)^2 \\
				+\frac 54 \ird W(u_n)\, dx-\frac{3\l_n}{8}V_2(u_n)=c_n+o_n(1).
			\end{multline}
		Now, multiplying \eqref{eq:neh} by $-\frac 3 8 $ and adding \eqref{eq:mix} we get
			\begin{multline*}
				\frac 18 \ird |\n u_n|^2\, dx +\frac 1{64\pi}\left(\ird u_n^2\,dx\right)^2 \\
				+\frac 54 \ird W(u_n)\, dx-\frac 3 8 \ird W'(u_n)u_n\,dx =c_n + o_n(1).
			\end{multline*}
		From this, $W2$ and Gagliardo-Nirenberg inequality, it follows
			\begin{align*}
				\frac 18 \ird |\n u_n|^2\, dx +\frac 1{64\pi}\left(\ird u_n^2\,dx\right)^2 & \le M + C(\|u_n\|_2^2+ \|u_n\|_{p}^{p})\\
				&\le M+C(\|u_n\|_2^2+\|\n u_n\|_2^{p-2}\|u_n\|_2^2)\\
				&\le M + C (1+\|\n u_n\|_2^{p-2})\|u_n\|_2^2.
			\end{align*}
		Now, applying the inequality $|ab|\le \eps a^2+C_\eps b^2$ which holds for every $\eps>0$ and a corresponding $C_\eps>0$, we have
			\begin{equation*}
				\frac 18 \ird |\n u_n|^2\, dx +\frac 1{64\pi}\left(\ird u_n^2\,dx\right)^2 \le M + C_\eps (1+\|\n u_n\|_2^{p-2})^2 + \eps \|u_n\|_2^4
			\end{equation*}
		which, for $\eps < \frac 1 {32}$ implies boundedness in $H^1(\RD)$.
		
		Now assume $W3'$. Multiply \eqref{eq:neh} by $-\frac 14$ and add \eqref{eq:func}. We immediately see that $(\|\n u_n\|)_n$ is bounded. Now, by \eqref{eq:poho2} and \eqref{eq:estV2}, we deduce that
			 \begin{equation*}
				 \|u_n\|_2^4\le C \|u_n\|_2^3\|\n u_n\|_2
			 \end{equation*}
		and then we prove boundedness in $H^1(\RD)$ also in this case.
		
		 Again as in \cite[Proof of  Prooposition 3.1]{CW}, we deduce that, up to translations and a suitable choise of a subsequence, $(u_n)_n$ strongly converges in $E$ to a function $\bar u\neq 0$. Of course $I'(\bar u)=0$ and we conclude.
 	\end{proof}
 
 	\section{The problem $\mathcal P_-$}\label{sec:P_-}
 	
 		\subsection{The nonautonomous case}
 			In this section we will introduce a variatonal approach in order to solve the problem $(\mathcal{P}_-)$. 
 			Define $X=\{u\in \H\mid \ird (1+|x|^\a)u^2\,dx<+\infty \}$ endowed with the norm
 				$$\|u\|=\left(\|\n u\|_2^2+\ird (1+|x|^\a)u^2\,dx\right)^{\frac 12}$$
 			induced by the scalar product 
 				$$(u\cdot v)=\ird \n u\n v\,dx+ \ird (1+ |x|^\a) uv\, dx.$$
 			We denote by $\|\cdot\|_*$ the following weighted $L^2$ norm
 				$$\|u\|_*=\left(\ird (1+|x|^\a)u^2\,dx\right)^{\frac 12},$$
 			so that $\|u\|^2=\|\n u\|_2^2+\|u\|_*^2$.\\
 			Observe that for every $\a>0$ there exists $C_\a>1$ such that, taken any $r\ge 0$, we have that 
 				\begin{equation}\label{eq:Ca}
					\log(2+r)\le C_\a +r^a.
 				\end{equation}
 			and then for every $u\in X$
 				\begin{equation}\label{inclusion}
	 				\ird\log (2+|x|) u^2\, dx\le \ird (C_\a+|x|^\a) u^2\, dx\le C_\a \|u\|^2_*.
 				\end{equation}
 			By using the same arguments as in \cite[Lemma 2.2]{CW}, we deduce that the functional
 			\begin{equation*}
 			I_\a (u)=\frac 12 \ird \big(|\n u|^2+(1+|x|^\a)u^2\big)\, dx-  \frac 1 4V_0(u)+\frac 1 p \ird |u|^p\, dx
 			\end{equation*}
 			is well defined and $C^1$ in $X$. Moreover its critical points are related with solutions of \eqref{eq:e2} as for $(\mathcal P_+)$. \\
 			In particular:

 				\begin{align}\label{eq:estimate}
	 				V_1(u)&=\frac 1{2\pi} \ird\ird \log(2+|x-y|) u^2(x) u^2(y)\, dx\,dy\\
	 								&\le \frac 1{2\pi}  \ird\ird \log(2+|x|) u^2(x) u^2(y)\, dx\,dy\nonumber\\
	 								&\qquad+\frac 1{2\pi} \ird\ird \log(2+|y|) u^2(x) u^2(y)\, dx\,dy\nonumber\\
	 								&=\frac 1{\pi} \ird \log(2+|x|) u^2(x)\, dx \ird u^2(y)\,dy\nonumber\\
	 								&\le\frac{C_\a}{\pi}   \ird (1+|x|^\a) u^2(x)\, dx \ird u^2(y)\,dy\nonumber\\
	 								&=\frac {C_\a}\pi \|u\|_2^2\|u\|_*^2.\nonumber
 				\end{align}

 \begin{proposition}
 	There exist $\rho>0$, $\beta > 0$ and $\bar u\in X$ such that, said $B_{\rho}$ the ball of radius $\rho$ in $E$,
 	\begin{itemize}
 		\item ${I_\a}_{|B_{\rho}}\ge 0$ and ${I_\a}_{|\partial B_{\rho}}\ge \beta$,
 		\item $\|\bar u\|>\rho$ and ${I_\a}(\bar u)<0$.
 	\end{itemize}
 \end{proposition}

\begin{proof}
	The first property follows by \eqref{eq:estimate} since, for every $u\in X$, we have
		\begin{equation*}
		I(u) \ge \frac 12 \|u\|^2  - \frac  {C_\a}{\pi} \|u\|^4 .
		\end{equation*}
	As to the second geometrical property, we consider $u\in X$ and test the functional in $u_t=t^r u(\cdot/t)$
		\begin{align*}
			I_\a(u_t)&= \frac{t^{2r}}{2}\|\n u\|_2^2 + \frac{t^{2r+2}}2\ird (1+t^\a|x|^a)u^2\, dx\\ 
			&\qquad - \frac{t^{4r+4}\log t}{8\pi}\left(\ird u^2\, dx \right)^2 - \frac{t^{4r+4}}{4} V_0(u) + \frac{t^{pr+2}}{p}\ird |u|^p\, d.x
		\end{align*}
	In order to have $\lim_{t\to + \infty}I_\a(u_t)=-\infty$ we look for a number $r>0$ such that 
		\begin{equation*}
			\left \{\begin{array}{l}
				2r+2+\a\le 4r+4\\
				pr+2\le 4r+4.
			\end{array}
			\right.
		\end{equation*}
	By simple computations one can see that such an $r>0$ exists for every $\a>0$ if $p\le 4$, otherwise, if $p>4$, we need that $\a\le 2 \left(\frac{p-2}{p-4}\right)$. 
\end{proof}
Define $c_\a$ as the mountain pass level of $I_\a$.
	\begin{lemma}\label{le:Cerami}
			There exists a Cerami sequence for $I_\a$ at the level $c_\a$, namely a sequence $(u_n)_n$ in $X$ such that
				\begin{align*}
					&I_\a(u_n)\to c_\a,\\
					&\|I'(u_n)\| (1+\|u_n\|)\to 0,
				\end{align*}
			for which $J(u_n)\to 0$ where 
				\begin{multline}
					J(u)=\frac{2}{p-4}\|\n u\|_2^2+\frac{p-2}{p-4}\|u\|_2^2+\left(\frac{p-2}{p-4}+\frac \a 2\right)\ird |x|^\a u^2\,dx \\-\frac 1 {8\pi} \|u\|_2^4- \frac{p-2}{p-4} V_0(u)+ 4\left(\frac{p-2}{p(p-4)}\right)\|u\|^p_p
				\end{multline}
			if $p>4$, and, for an arbitray positive $r$ such that $r\ge\frac{\a-2}{2},$
				\begin{multline}
					J(u)=r\|\n u\|_2^2+(r+1)\|u\|_2^2+\left(r+1+\frac \a 2\right)\ird |x|^\a u^2\,dx \\-\frac 1 {8\pi} \|u\|_2^4- (r+1) V_0(u)+ \frac{pr+2}{p}\|u\|^p_p
				\end{multline}
			if $2< p\le 4$.
	\end{lemma}
	\begin{proof}
		The proof is analogous to that provided in \cite[Lemma 3.2]{DW}, replacing the continuous map $\rho$ with
			\begin{equation}\label{eq:eta}
				\eta:\R\times X\to X,\quad \eta(t,v)(x)= 
						\left \{
						\begin{array}{ll}
							e^{\frac{2t}{p-4}}v(e^{-t}x)&\hbox{if } p> 4\\
							e^{rt}v(e^{-t}x)&\hbox{if } 2< p \le 4,
						\end{array}
						\right.
			\end{equation}
		where $r>0$ and $r\ge\frac{\a-2}{2}$.
		Indeed, if we set $\varphi=I_\a\circ \eta$, then for every $(t,v)\in\R\times X$ we have
			\begin{multline*}
				\varphi(t,v)=\frac{e^{\frac{4t}{p-4}}}{2}\|\n v\|_2^2 + \frac{e^{\frac{2pt -4t}{p-4}}}2\ird (1+e^{\a t}|x|^a)v^2\, dx\\ 
				\qquad - \frac{te^{\frac{4pt-8t}{p-4}}}{8\pi}\left(\ird v^2\, dx \right)^2 - \frac{e^{\frac{4pt-8t}{p-4}}}{4} V_0(v) + \frac{e^{\frac{4pt-8t}{p-4}}}{p}\ird |v|^p\, dx
			\end{multline*}
		if $p>4$, and 
			\begin{multline*}
				\varphi(t,v)=\frac{e^{2rt}}{2}\|\n v\|_2^2 + \frac{e^{(2r+2)t}}2\ird (1+e^{\a t}|x|^a)v^2\, dx\\ 
				\qquad - \frac{te^{(4r+4) t}}{8\pi}\left(\ird v^2\, dx \right)^2 - \frac{e^{(4r+4) t}}{4} V_0(v) + \frac{e^{(pr+2)t}}{p}\ird |v|^p\, dx
			\end{multline*}
		if $2< p \le 4$.
		
					In any case, since for every $v\in X$ we have $\lim_{t\to + \infty}\varphi(t,v)=-\infty$, the class of paths 
					$$\tilde\G:\{\tilde\g:[0,1]\to\R\times X\mid \tilde\g \hbox{ is continuous}, \tilde \g(0)=(0,0),  \varphi(\tilde\g(1))<0\}$$
					is nonempty. This permits to repeat the arguments in \cite[Lemma 3.2]{DW} and conclude.
	\end{proof}

	 \begin{proof}[Proof. of Theorem \ref{main2}]
			First we are going to prove that there exists a bounded Cerami sequence for the functional $I_\a$ at the level $c_\a$.\\
		Consider $(u_n)_n$ the sequence obtained in Lemma \ref{le:Cerami} and prove that it is bounded in $X$.
		
		We distinguish two cases.
		\begin{itemize}
			\item[{\it $1^{st}$ case}:] $p>4$. Then 
			
			\begin{align}
			&\frac 12 \ird \big(|\n u_n|^2+(1+|x|^\a)u_n^2\big)\, dx-  \frac 1 4V_0(u_n)+\frac 1 p \ird |u_n|^p\, dx= c_\a +o_n(1)\label{fun}\\
			&\ird \big(|\n u_n|^2+(1+|x|^\a)u_n^2\big)\, dx-  V_0(u_n)+ \ird |u_n|^p\, dx= o_n(1)\label{ne}\\
			&\frac{2}{p-4}\|\n u_n\|_2^2+\frac{p-2}{p-4}\|u_n\|_2^2+\left(\frac{p-2}{p-4}+\frac \a 2\right)\ird |x|^\a u_n^2\,dx\label{po} \\
			&\qquad\qquad-\frac 1 {8\pi} \|u_n\|_2^4- \frac{p-2}{p-4} V_0(u_n)+ 4\left(\frac{p-2}{p(p-4)}\right)\ird |u_n|^p\, dx=o_n(1).\nonumber
			\end{align} 
			Multiplying \eqref{po} by $-\frac{p-4}{4(p-2)}$ and adding \eqref{fun} we obtain 
				\begin{multline}\label{eq:bou1}
					\frac{p-3}{2(p-2)}\|\n u_n\|_2^2+\frac 14 \|u_n\|_2^2+\frac{2(p-2)-(p-4)\a}{8(p-2)}\ird |x|^\a u_n^2\, dx\\
					 + \frac{p-4}{32\pi(p-2)}\|u_n\|_2^4=c_\a + o_n(1)
				\end{multline}
			
			\item[{\it $2^{nd}$ case}:] $2<p\le4$. Then, we again have \eqref{fun} and, instead of \eqref{po},
				\begin{multline}\label{po2}
					r\|\n u_n\|_2^2+(r+1)\|u_n\|_2^2+\left(r+1+\frac \a 2\right)\ird |x|^\a u_n^2\,dx \\
					-\frac 1 {8\pi} \|u_n\|_2^4 - (r+1) V_0(u_n)+ \frac{pr+2}{p}\|u_n\|^p_p=o_n(1)
				\end{multline}
			where $r>\max (0,(\a-2/2))$.\\ 
			Multiplying \eqref{po2} by $-\frac{1}{4(r+1)}$ and adding \eqref{fun} we obtain
				\begin{multline}\label{eq:bou2}
				\frac{r+2}{4(r+1)}\|\n u_n\|_2^2+\frac 14 \|u_n\|_2^2+\frac{2(r+1)-\a}{8(r+1)}\ird |x|^\a u_n^2\, dx\\
				+ \frac{1}{32\pi(r+1)}\|u_n\|_2^4+\frac{(4-p)r+2}{4p(r+1)}\|u_n\|_p^p=c_\a + o_n(1).
				\end{multline}				
		\end{itemize}
	
	From \eqref{eq:bou1} and \eqref{eq:bou2} we deduce that the sequence $(u_n)_n$ is in any case bounded in $\H$.\\ 
	Now, if $2<p\le 4$, from \eqref{eq:bou2} we soon deduce that $(u_n)_n$ is bounded also in the norm $\|\cdot\|_*$.\\ 
	If $p>4$, comparing \eqref{fun} with \eqref{ne} in order to eliminate $V_0$, we get  
		\begin{align*}
			\frac 14 \ird \big(|\n u_n|^2+(1+|x|^\a)u_n^2\big)\, dx&=  c_\a +\frac{p-4}{4p} \ird |u_n|^p\, dx+o_n(1)\\
			&\le C+o_n(1),
		\end{align*}
	and then, again $(\|u_n\|_*)_n$ is bounded.
	
	Now, up to a subsequence, we are allowed to assume that there exists $\bar u \in X$ such that 
		\begin{align}
			&u_n\rightharpoonup\bar u\hbox{ in }X,\label{eq:weak}\\
			&u_n\to\bar u\hbox{ in } L^q(\RD),\hbox{ for any } q\ge 2.\label{eq:strong}
		\end{align}
	We are going to show that $u_n\to \bar u$ strongly in $X$. 
	
	First observe that, since $\langle I'_\a(u_n),u_n-\bar u\rangle \to 0$, by \eqref{eq:weak} and \eqref{eq:strong}, we have
	 \begin{align}\label{eq:norms}
		 o_n(1)=\|u_n\|^2-\|\bar u\|^2 - \frac 14 \left[V'_1(u_n)(u_n-\bar u)-V'_2(u_n)(u_n-\bar u)\right].
	 \end{align}
	 By computations based on the application of Hardy-Littlewood-Sobolev inequality, we have
	 	\begin{equation}\label{eq:estV'2}
	 		|V'_2(u_n)(u_n-\bar u)|\le C \|u_n\|^3_{\frac 83}\|u_n-\bar u\|_{\frac 83}=o_n(1).
	 	\end{equation}
	 On the other hand, by \eqref{eq:Ca} and \eqref{eq:strong},
	 	\begin{align}\label{eq:estV'1}
		 	|V'_1(u_n)(u_n-\bar u)|&=\frac 1 {2\pi}\left|\ird\ird \log (2+|x-y|) u^2_n(x)u_n(y)(u_n(y)-\bar u(y))\, dx dy\right|\\
	 								&\le \frac 1{2\pi} \left| \ird\ird \log(2+|x|) u_n^2(x)u_n(y) (u_n(y)-\bar u(y))\, dx\,dy\right|\nonumber\\
	 								&\qquad+\frac 1{2\pi}\left| \ird\ird \log(2+|y|) u_n^2(x) u_n(y)(u_n(y)-\bar u(y))\, dx\,dy\right|\nonumber\\
	 								&\le \frac 1{2\pi} \left|\ird\ird (C_\a+|x|^\a) u_n^2(x) u_n(y)(u_n(y)-\bar u(y))\, dx\,dy\right|\nonumber\\
	 								&\qquad+\frac 1{2\pi}\|u_n\|_2^2 \left|\ird \log(2+|y|) u_n(y)(u_n(y)-\bar u(y))\,dy\right|\nonumber\\
									&\le\frac {C_\a}{2\pi} \|u_n\|_*^2\|u_n\|_2\|u_n-\bar u\|_2\nonumber\\
									&\qquad+\frac 1{2\pi}\|u_n\|_2^2\left| \ird \log(2+|y|) u_n(y)(u_n(y)-\bar u(y))\,dy\right|\nonumber\\
									&=o_n(1)+\frac 1{2\pi}\|u_n\|_2^2 \left|\ird \log(2+|y|) u_n(y)(u_n(y)-\bar u(y))\,dy\right|.\nonumber
	 	\end{align}
	 	Take $C'_\a>1$ such that for any $r>0$
	 			$$\log^2(2+r)\le C'_\a + r^\a.$$ 
	 	Again by \eqref {eq:strong}, we have
	 		\begin{align*}
				\left|\ird \log(2+|y|) u_n(y)(u_n(y)-\bar u(y))\,dy\right|&\le \left(\ird \log^2(2+|y|) u_n^2(y)\, dy\right)^{\frac 12}\|u_n-\bar u\|_2\\
				&\le\left(\ird (C_\a'+|x|^\a)) u_n^2(y)\, dy\right)^{\frac 12}\|u_n-\bar u\|_2\\
				&=o_n(1),
	 		\end{align*}
	 	and then, by \eqref{eq:estV'1}, 
	 		\begin{equation}\label{eq:V'1bis}
	 			V'_1(u_n)(u_n-\bar u)=o_n(1).
	 		\end{equation}
	 	Putting together \eqref{eq:norms}, \eqref{eq:estV'2}  and \eqref{eq:V'1bis}, by \eqref{eq:weak} we deduce that $u_n\to \bar u$ in $X$. Since $(u_n)_n$ is a Cerami sequence for $I_\a$ at the mountain pass level $c_\a$, we conclude that $\bar u$ is a nontrivial solution to \eqref{eq:e1}.
	 \end{proof}
 
 \subsection{The autonomous case}
 	In this section we assume that the hypotheses in Theorem \ref{main3} are satisfied. A large part of the arguments are similar to those in the previous subsection, so we will sketch the proof framework.
 	
 	Consider the space $\Ha$, defined as the closure of $C_0^{\infty}(\RD)$ with respect to
the norm $\|\cdot\|$ defined by
 	 	\begin{equation*}
 	 		\|\cdot\|=\|\n \cdot\|_2+ \|\cdot\|_\b
 	 	\end{equation*}
and denote by $\Har$ the subspace of radial functions. By a standard method (see for example \cite{Br}) we can see that $\Ha\hookrightarrow L^q(\RD)$ for every $q\ge \beta$. Moreover, since by Strauss radial Lemma \cite{St}, there exist $C$ and $C'$ positive constants such that for every $u\in\Har$ 
 		\begin{equation}
			|u(x)|\le \frac C{|x|}\| u\|_{H^1}\le \frac {C'}{|x|}\| u\|,\quad |x|\ge 1,
 		\end{equation}
usual arguments lead to conclude that the embedding $\Har\hookrightarrow L^q(\RD)$ is compact for $q>\beta$.
Moreover, 
	\begin{align}\label{eq:incl}
		\ird \log (2+|x|)u^2\, dx&\le \log 3\int_{B_1} u^2\, dx\\
		&\qquad+{C'}^{2-\b}\|u\|^{2-\b}\int_{\RD\setminus B_1} \frac{\log(2+|x|)}{|x|^{2-\b}}|u|^\b \, dx\nonumber\\
		&\le C\left(\|u\|_2^2+ \|u\|^{2-\b}\|u\|_{\b}^{\b}\right)\le C \|u\|^2\nonumber
	\end{align}
 	 
and then the functional 
 	 		$$I_\b(u)=\frac 12 \ird |\n u|^2 \, dx -  \frac 1 4V_0(u)+\ird \left(\frac 1 \b |u|^\b+\frac 1 p |u|^p\right)\, dx$$
 	 		is well defined and $C^1$ in $\Har$. Finally, any critical point of the functional yields a solution to \eqref{eq:e2}.

\begin{proposition}
	There exist $\rho>0$, $\g > 0$ and $\bar u\in \Har$ such that, said $B_{\rho}$ the ball of radius $\rho$ in $\Har$,
	\begin{itemize}
		\item ${I_\b}_{|B_{\rho}}\ge 0$ and ${I_\b}_{|\partial B_{\rho}}\ge \g$,
		\item $\|\bar u\|>\rho$ and ${I_\b}(\bar u)<0$.
	\end{itemize}
\end{proposition}
\begin{proof}
	The first property is a consequence of \eqref{eq:incl}. Indeed, computing as in \eqref{eq:estimate} we have $V_1(u)\le C \|u\|^4$ for all $u\in\Har$ and then, if $\|u\|$ small enough,
		$$I_\b(u)\ge \frac 1 \b \|u\|^2-C\|u\|^4.$$
	As to the second property,  we consider $u\in \Har$ and test the functional in $u_t= u(\cdot/t)$
	\begin{align*}
	I_\b(u_t)= \frac{1}{2}\|\n u\|_2^2   - \frac{t^{4}\log t}{8\pi}\left(\ird u^2\, dx \right)^2 - \frac{t^{4}}{4} V_0(u) + t^{2}\ird \left(\frac{|u|^\b}\b+\frac{|u|^p}p\right)\, dx.
	\end{align*}
	Then  $\lim_{t\to + \infty}I_\b(u_t)=-\infty$ and we conclude.
\end{proof}
Define $c_\b$ as the mountain pass level for $I_\b$.
 	 		\begin{lemma}\label{le:Cerami2}
 	 		There exists a Cerami sequence for $I_\b$ at the level $c_\b$, namely a sequence $(u_n)_n$ in $X$ such that
 	 		\begin{align*}
 	 		&I_\b(u_n)\to c_\b,\\
 	 		&\|I'(u_n)\| (1+\|u_n\|)\to 0,
 	 		\end{align*}
 	 		for which $P(u_n)\to 0$ where 
 	 		\begin{equation}
 	 		P(u)=-\frac 1 {8\pi} \|u\|_2^4-  V_0(u)+ \frac 2\b \|u\|_\b^\b+\frac 2 p\|u\|_p^p.
 	 		\end{equation}
 	 	\end{lemma}
 	 	\begin{proof}
 	 		The proof is definitely the same as in \cite{HIT}.
 	 	\end{proof}

  		 	Now we are ready for the following
	 \begin{proof}[Proof of Theorem \ref{main2}]
	 	Consider a sequence $(u_n)_n$ as in Lemma \ref{le:Cerami2}. Then we have
	 				\begin{align}
	 	&\frac 12 \ird |\n u_n|^2\, dx-  \frac 1 4V_0(u_n)+\ird \left(\frac 1 \b |u_n|^\b+\frac 1 p |u_n|^p\right)\, dx= c_\b +o_n(1)\label{fun2}\\
	 	&\ird |\n u_n|^2\, dx-  V_0(u_n)+ \ird (|u_n|^\b+|u_n|^p)\, dx= o_n(1)\label{ne2}\\
	 	&-\frac 1 {8\pi} \|u_n\|_2^4-  V_0(u_n)+2 \ird \left(\frac 1 \b |u_n|^\b+\frac 1 p |u_n|^p\right)\, dx=o_n(1).\label{po3}
	 	\end{align} 
	Comparing \eqref{fun2} and \eqref{po3} in order to delete $V_0(u_n)$ we get
		\begin{equation*}
			\frac12 \ird |\n u_n|^2\, dx + \frac 1{32\pi} \|u_n\|_2^4 + \frac 12 \ird \left(\frac 1 \b |u_n|^\b+\frac 1 p |u_n|^p\right)\, dx= c_\b+o_n(1)
		\end{equation*}
	and then we deduce the boundedness of $(u_n)_n$ in $\Har$. \\
	Up to a subsequence, we are allowed to assume that there exists $\bar u \in \Har$ such that 
	\begin{align*}
	&u_n\rightharpoonup\bar u\hbox{ in }\Har,\\
	&u_n\to\bar u\hbox{ in } L^q(\RD),\hbox{ for any } q>\b.
	\end{align*}  
	Now, the proof proceeds analogously to that of  Thoerem \ref{main2}, using  \eqref{eq:incl} in the place of  \eqref{inclusion}, and we arrive to show that $u_n\to \bar u$ in $\Har$. Since $(u_n)_n$ is a Cerami sequence at the mountain pass level, this is sufficient to conclude that $\bar u$ is a nontrivial solution.
	 \end{proof}

	\subsection{A nonexistence result in $\RT$}
	
	First we present the following Pohozaev identity, whose proof can be found in \cite{DM2}
		\begin{lemma}
			Under the assumptions of Theorem \ref{nonex}, if  $(u,\phi)\in H^1(\RT)\times \D$  solves $\mathcal P_-$ in $\RT$ and $W(u)\in L^1(\RT)$ then the following identity holds
				\begin{equation}\label{pohoid}
					\frac 12 \irt |\n u|^2\, dx +\frac 54 \irt \phi u^2 \, dx+ 3 \irt W(u)\, dx =0.
				\end{equation}
		\end{lemma}
	As an immediate consequence we have the following
		\begin{proof}[Proof of Theorem \ref{nonex}]
			As it is well known, since $\phi\in \D$ solves the second equation, it can be explicitly expressed by
				\begin{equation*}
					\phi(x)= \frac 1 {4\pi} \irt \frac{u^2(y)}{|x-y|}\, dy\ge 0.
				\end{equation*} 
			As a consequence, by \ref{pohoid} and $W2$, 
				$$0\le\frac 12 \irt |\n u|^2\, dx = - \frac 54 \irt \phi u^2\, dx - 3 \irt W(u)\, dx \le 0$$
			and then $u=0$. Of course, this implies that also $\phi=0$ and we conclude.
		\end{proof}

\end{document}